\documentclass[10pt]{amsart}   
\usepackage{amsmath,amsthm}

\usepackage[english]{babel}
\usepackage{mathtools}
\usepackage[T1]{fontenc} %encoding
\usepackage[scr=boondoxo]{mathalfa}   %fontcaligrafica 
\usepackage{dsfont}
\usepackage[utf8]{inputenc} %acentos
\usepackage[all]{xy}

\newtheorem{theorem}{Theorem}[section]
\newtheorem*{theorem*}{Theorem}

\newtheorem{proposition}{Proposition}[section]

\theoremstyle{definition}
\newtheorem{definition}{Definition}[section]

\newtheorem*{question}{Question}

\usepackage{hyperref}
\hypersetup{
	unicode=true,          % non-Latin characters in Acrobat?s bookmarks
	pdftoolbar=true,        % show Acrobat?s toolbar?
	pdfmenubar=true,        % show Acrobat?s menu?
	pdffitwindow=false,     % window fit to page when opened
	pdfstartview={FitH},    % fits the width of the page to the window
	pdftitle={Geometrical constructions of equilibrium states},    % title
	pdfauthor={Pablo D. Carrasco and Federico Rodriguez-Hertz},     % author
	pdfkeywords={Equilibrium States, Conditional Measures, Unique Ergodicity}, % list of keywords
	pdfnewwindow=true,      % links in new PDF window
	colorlinks=true,       % false: boxed links; true: colored links
	linkcolor=blue,          % color of internal links (change box color with linkbordercolor)
	citecolor=green,        % color of links to bibliography
	filecolor=magenta,      % color of file links
	urlcolor=cyan           % color of external links
}

\newcommand{\muf}[1][\varphi]{\mathrm{m}_{\scriptstyle #1}}

\newcommand{\mux}[1][x]{\mu^u_{#1}}

\newcommand{\msx}[1][x]{\mu^s_{#1}}

\newcommand{\mcux}[1][x]{\mu^{cu}_{#1}}

\newcommand{\mcsx}[1][x]{\mu^{cs}_{#1}}

\newcommand{\nux}[1][x]{\nu^u_{#1}}

\newcommand{\Jacu}[1][x_0,y_0]{\mathrm{Jac}^{u}_{#1}}

\newcommand{\mm}{\mathrm{m}}

\newcommand{\norm}[1]{\|#1\|}

\DeclarePairedDelimiterX\abs[1]||{\ifblank{#1}{\:\cdot\:}{#1}}   %%valor absoluto
\DeclarePairedDelimiterX\prodi[2]{\langle}{\rangle}{#1,#2}          %%producto interno
\DeclarePairedDelimiterXPP\prodiL[2]{}\langle\rangle{_{\scaleto{\mathscr{L}^{2}}{5pt}}}{#1,#2}         %
\DeclarePairedDelimiterXPP\prodiLL[3]{}\langle\rangle{_{\scaleto{\mathscr{L}^{2}}{7pt}(#3)}}{#1,#2}

\DeclarePairedDelimiterXPP\normOP[1]{}\|\|{_{\scaleto{\mathtt{OP}}{4pt}}}{#1}  %%norma Operador

\DeclarePairedDelimiterXPP\normH[1]{}\|\|{_{\scaleto{\Hil}{5pt}}}{#1}            %%norma Hilbert

\DeclarePairedDelimiterXPP\normTV[1]{}\|\|{_{\scaleto{TV}{5pt}}}{#1} 

\DeclarePairedDelimiterXPP\dis[3]{\mathtt{d}_{\scriptstyle #1}}(){}{#2,#3}  %%distancia 

\DeclarePairedDelimiterXPP\normC[2]{}\|\|{_{\scaleto{\mathcal{C}^{#2}}{5pt}}}{#1}  %%norma cr 
\DeclarePairedDelimiterXPP\normc[2]{}\|\|{_{\scaleto{\mathcal{c}^{#2}}{5pt}}}{#1}  %%norma cr
\DeclarePairedDelimiterXPP\normL[2]{}\|\|{_{\scaleto{\mathscr{L}^{#2}}{5pt}}}{#1}  %%norma Lp
\DeclarePairedDelimiterXPP\norml[2]{}\|\|{_{\scaleto{\ell^{#2}}{4pt}}}{#1}  %%norma lp
\DeclarePairedDelimiterX\Set[1]\{\}{#1}     %Conjunto

\DeclarePairedDelimiterXPP\Prob[1]{\mathbb{P}}(){}{#1} %prob condicional 

\DeclarePairedDelimiterXPP\ie[2]{\mathbb{E}_{#1}}(){}{#2}

\DeclareFontFamily{U}{mathb}{\hyphenchar\font45}
\DeclareFontShape{U}{mathb}{m}{n}{
      <5> <6> <7> <8> <9> <10> gen * mathb
      <10.95> mathb10 <12> <14.4> <17.28> <20.74> <24.88> mathb12
      }{}
\DeclareSymbolFont{mathb}{U}{mathb}{m}{n}
\DeclareMathSymbol{\toit}{3}{mathb}{'375}

\makeatletter
\def\referencia#1#2{\begingroup
    #2%
    \def\@currentlabel{#2}%
    \phantomsection\label{#1}\endgroup
}
\makeatother

\newcommand{\Real}{\mathbb R}       %REALES
\newcommand{\Nat}{\mathbb N}        %NATURALES 
\newcommand{\Tor}{\mathbb T}        %TORO 
\newcommand{\Z}{\mathbb Z}          %Enteros 
\newcommand{\oo}{\infty}            %INFINITO

\newcommand{\PTM}[2]{\ensuremath{\mathscr{Pr}_{\cramped[\scriptstyle]{#1}}(#2)}}      	%PROBABILIDADES T-INV EN M
\newcommand{\ETM}[2]{\ensuremath{\mathscr{Erg}_{#1}(#2)}}    				%MEDIDAS T-ERGODICAS EN M
\newcommand{\aep}{\ensuremath{\operatorname{-}a.e.}}
\newcommand{\mae}{\ensuremath{\mu\operatorname{-}a.e.}}

\newcommand{\entPn}[1][P]{\ensuremath{H_m((P)_0^{-n+1})}}

\newcommand{\SB}[1][\varphi]{S_n#1}

\newcommand{\Lp}[1][2]{\ensuremath{\mathcal{L}^{#1}}}

\DeclarePairedDelimiterXPP\clo[1]{\mathtt{cl}}(){}{#1}				%clausura cl
\newcommand{\der}{\operatorname{d}\hspace{-2pt}}

\newcommand{\Hil}{\mathcal{H}}

\newcommand{\F}{\ensuremath{\mathcal{F}}}
\newcommand{\Fc}{\ensuremath{\mathcal{W}^c}}
\newcommand{\Fs}{\ensuremath{\mathcal{W}^s}}

\newcommand{\Fu}{\ensuremath{\mathcal{W}^u}}

\newcommand{\Fcs}{\ensuremath{\mathcal{W}^{cs}}}
\newcommand{\Fcu}{\ensuremath{\mathcal{W}^{cu}}}

\newcommand{\Ws}[1]{\ensuremath{W^s(#1)}}
\newcommand{\Wu}[1]{\ensuremath{W^u(#1)}}
\newcommand{\Wc}[1]{\ensuremath{W^c(#1)}}
\newcommand{\Wcs}[1]{\ensuremath{W^{cs}(#1)}}

\newcommand{\Ben}[1][x]{D(#1,\epsilon,n)}		%Bola de Bowen
\newcommand{\Ptop}{P_{\scriptstyle \mathit{top}}}			%Presion
\newcommand{\htop}{h_{\scriptstyle \mathit{top}}}   %entropia

\newcommand{\hu}{\mathrm{hol}^u}

\newcommand{\al}{\alpha}

\newcommand{\ep}{\epsilon}
\begin{document}

%%%%%%%%%%%%%%%%%%%%%%%%%%%%%%%%%%%%%%%%%%%%%%%%%%%%%%%%%%%%%%%%%%%%%%%%%%%%%%%%%%%%%%%%%

	\title{Geometrical constructions of equilibrium states}
	\author{Pablo D. Carrasco}
	\address{ICEx-UFMG, Av. Ant\^{o}nio Carlos, 6627 CEP 31270-901, Belo Horizonte, Brazil}
	\email{pdcarrasco@mat.ufmg.br}

	\author{Federico Rodriguez-Hertz}
	\address{Penn State, 227 McAllister Building, University Park, State College, PA16802}
	\email{hertz@math.psu.edu}
	% \author[1]{Pablo D. Carrasco \thanks{pdcarrasco@mat.ufmg.br}}
	% \author[2]{Federico Rodriguez-Hertz \thanks{hertz@math.psu.edu}}
 %    \affil[1]{ICEx-UFMG, Avda. Presidente Antonio Carlos 6627, Belo Horizonte-MG, BR31270-90}
	% \affil[2]{Penn State, 227 McAllister Building, University Park, State College, PA16802}

	\date{\today}

	\begin{abstract}
    In this note we report some advances in the study of thermodynamic formalism for a class of partially hyperbolic system -- center isometries, that includes regular elements in Anosov actions. The techniques are of geometric flavor (in particular, not relying in symbolic dynamics) and even provide new information in the classical case. 

    For such systems, we give in particular a constructive proof of the existence of the SRB measure and of the entropy maximizing measure. It is also established very fine statistical properties (Bernoulliness), and it is given a characterization of equilibrium states in terms of their conditional measures in the stable/unstable lamination, similar to the SRB case. The construction is applied to obtain the uniqueness of quasi-invariant measures associated to H\"older Jacobian for the horocyclic flow.
    \end{abstract}

	\maketitle

\section{Introduction} % (fold)
\label{sec:introduction}

Let $M$ be a closed Riemannian manifold and $f:M\to M$ be a diffeomorphism. The study of the statistical properties of $f$, and in particular its invariant measures, provides a powerful tool to understand the dynamical (but, not only) properties of the map. Let us recall that a Borel probability measure $\mu$ on $M$ is $f$-invariant if $f_{\ast}\mu=\mu$, where $f_{\ast}\mu(A)=\mu(f^{-1}A)$, and denote by $\PTM{f}{M}$ the set of all such measures on $M$.

For maps having rich dynamics, the set $\PTM{f}{M}$ is usually very complicated, so further restrictions have to be imposed to obtain meaningful results. A particularly important choice, both from the theoretical and applied point of view, are the so called equilibrium states; these are $f$-invariant measures obtained by a variational principle associated to some real valued map. See part \ref{sub:basic_thermodynamic_formalism} for the precise definition.

The study of such types of measures and its properties (thermodynamic formalism) is notably well developed for completely hyperbolic systems (Anosov, or more generally, Axiom A). The reader can consult \cite{EquSta} for a introduction to these topics. Notwithstanding this, under very mild relaxations of the hyperbolicity hypothesis, the panorama becomes much less understood, even in the partially hyperbolic case, one of the most extensively researched type of systems, besides hyperbolic ones.

In the present note we announce an advance in this theory and report on new methods to study some natural classes of partially hyperbolic systems, as are the ones determined by (regular elements of) Anosov Lie group actions.

\section{Measures along leaves of invariant foliations} % (fold)
\label{sec:measures_along_leaves_of_invariant_foliations}

% section measures_along_leaves_of_invariant_foliations (end)
\subsection{Basic thermodynamic formalism} % (fold)
\label{sub:basic_thermodynamic_formalism}

Consider a compact metric space $(M,d)$, and let $f:M\to M$ be continuous map. Given $x\in M,\epsilon>0, n\in \Nat$ we denote by $D(x,\ep)$ the open disc of center $x$ and radius $\ep$, and by $\Ben$ the open $(\epsilon,n)$-Bowen ball centered at $x$,
\[
\Ben=\{y\in M:d(f^jx,f^jy)<\epsilon,\ j=0,\ldots,n-1 \};
\]
denote
\begin{equation*}
s(\ep,n)=\inf\{\#E:M=\cup_{x\in E}D(x,\epsilon,n)\}. 
\end{equation*}
\begin{definition}
The topological entropy of $f$ is the quantity
\[
	\htop(f)=\lim_{\epsilon\mapsto 0}\limsup_{n\mapsto\oo}\frac{\log s(\epsilon,n)}{n}.
\]
\end{definition}
Topological entropy is perhaps the most important topological invariant for continuous maps, measuring (loosely speaking) the exponential rate of expansion between orbits. We refer the reader to \cite{ErgWal} for a introduction to this theory, and the proof of the facts below.

\smallskip 
 
It so happens that it is also important to consider some weighted versions of the previous quantity. For $\varphi:M\to\Real$ a continuous map (called the \emph{potential} in this theory), denote 
\[
S(\epsilon,n)=\inf\{\sum_{x\in E}e^{\SB(x)}:M=\cup_{x\in E}D(x,\epsilon,n)\}\quad \SB=\sum_{i=0}^{n-1}\varphi\circ f^i.
\]
\begin{definition}
The topological pressure associated to the system $(f,\varphi)$ is 
\[
P_{\mathit{top}}^f(\varphi)=\lim_{\epsilon\mapsto 0}\limsup_{n\mapsto\oo}\frac{\log S(\epsilon,n)}{n}.
\]
\end{definition}
Observe that $\htop(f)=\Ptop(f,0)$. There is also a metric version of the previous concepts; to state it let us recall that $\mu\in \PTM{f}{M}$ is ergodic ($\mu\in \ETM{f}{M}$) if any $\psi\in\Lp[2](M,\mu)$ satisfying $f\circ \psi=\psi \aep$ is constant almost everywhere.

\begin{definition}
Let $\mu$ be an ergodic measure. Then the metric entropy of $f$ with respect to $\mu$ is 
\[
	h_{\mu}(f)=\lim_{\ep\to0}\liminf_{n\to\oo}\frac{-\log \Ben}{n},
\]
and if $\varphi$ is a potential, the metric pressure of $(f,\varphi)$ with respect to $\mu$ is
\[
	P_{\mu}(f,\varphi)=h_{\mu}(f)+\int \varphi \der \mu.
\]
\end{definition}

The limit above exists $\mae$, \cite{brinkatok}. We then have:

\begin{theorem*}[Variational Principle - Walters]
	It holds
	\[
	\Ptop(f,\varphi)=\sup_{\mathclap{\mu\in \ETM{f}{M}}}\ P_{\mu}(\varphi).
	\]
\end{theorem*}

\begin{definition}
$\mu\in\ETM{f}{M}$ is an (ergodic) equilibrium state if $\Ptop(f,\varphi)=P_{\mu}(f,\varphi)$.
\end{definition}

\subsection{Center isometries} % (fold)
\label{sub:center_isometries}

We now specify the type of map to which our results apply. 

\begin{definition}\label{def:centerisometry}
	Let  $M$\ be a closed manifold. A  diffeomorphism  $f:M\to M$ is a center isometry if there exist a continuous splitting of the tangent bundle of the form $TM=E^u\oplus E^c\oplus E^s$ and a (at least continuous) Riemannian metric $\norm{\cdot}$ such that
	\begin{enumerate}
		\item $\dim E^u, \dim E^s\geq 1$;
		\item $E^u,E^s,E^c$\ are $Df$-invariant in the sense that for every $x\in M, Df(E^{\ast}_x)=E^{\ast}_{fx}$;
		\item for every $x\in M$, for every unit vector $v^{\ast}\in E^{\ast}_x$, $\norm{D_xf(v^{c})}=1$ and $\norm{D_xf^n(v^{s})},\norm{D_xf^{-1}(v^{u})}<1.$
	\end{enumerate}
\end{definition}

Typical examples of such maps are group extensions of Anosov systems, and regular elements of Anosov actions (\cite{Katok1996}). We refer the reader to the survey article \cite{PartSurv} for basic information on these systems, and in particular for a discussion of the following:
\begin{itemize}
	\item[($\ast$)] all bundles $E^s, E^c, E^c, E^{cs}=E^c\oplus E^s, E^{cu}=E^c\oplus E^u$ are integrable to foliations $\Fs,\Fc,\Fu,\Fcs,\Fcu$ that are invariant under $f$, that is, $f$ permutes their leaves. Moreover, $\Fcs,\Fcu$ are sub-foliated by leaves of $\Fs,\Fc$ and $\Fu,\Fc$, respectively. Finally, $\Fs$ is exponentially contracting in the sense that $f$ exponentially contracts the corresponding induced distances on leaves, whereas $\Fu$ is exponentially expanding.  
\end{itemize}
% subsection center_isometries (end)

Our first central result is the next theorem.

\begin{theorem}\label{thm:familiasmedidas}
	Let $f:M\rightarrow M$ be a center isometry of class $\mathcal{C}^2$ such that every leaf of $\Fs,\Fu$ is dense. Let $\varphi:M\rightarrow \Real$ a H\"older potential that is either
    \begin{enumerate}
    \item constant on leaves of $\Fc$, or
    \item $\varphi=-\log \det Df|E^u$ (SRB case).
    \end{enumerate}
	 Then there exist  $P\in\Real$ and families of measures $\mu^u=\{\mux\}_{x\in M},\mu^s=\{\msx\}_{x\in M},\mu^{cu}=\{\mcux\}_{x\in M},\mu^{cs}=\{\mcsx\}_{x\in M}$ satisfying for every $x\in M$,
	\begin{enumerate}
		%\item The probability $\muf$ is the unique equilibrium state for the potential $\varphi$.
		\item the measure $\mu^{\ast},\ast\in \{u,s,cu,cs\}$ is a Radon measure on $W^{\ast}(x)$ of full (relative) support, and $y\in W^{\ast}(x)$ implies $\mu^{\ast}_x=\mu^{\ast}_y$; 
		\item the following quasi-invariance properties are satisfied
		\begin{enumerate}
			\item $\mu^{\sigma}_{fx}=e^{P-\varphi}f_{\ast}\mu^{\sigma}_{x}\quad \sigma\in\{u,cu\}$.
			\item $\mu^{\sigma}_{fx}=e^{\varphi-P}f_{\ast}\mu^{\sigma}_{x}\quad\sigma\in\{s,cs\}$. 
		\end{enumerate}
	\end{enumerate} 	 
\end{theorem}		

The method of the proof used is a generalization of the one used by Margulis for studying the entropy maximizing measure ($\varphi\equiv 0$) in mixing hyperbolic flows \cite{TesisMarg}.

\smallskip 
 A consequence of the above is that the foliation $\Fu$ is absolutely continuous with respect to the family $\mu^{cs}$. To explain this, consider $x_0,y_0$ in the same leaf of $\Fu$ and let $\hu=\hu_{x_0,y_0}:A(x_0)\subset \Wcs{x_0}\rightarrow B(y_0)\subset\Wcs{y_0}$ is the Poincaré map that sends $x_0$ to $y_0$. Then 
	\[
		(\hu)^{-1}_{\ast}\mu^{cs}_{y_0}=\Jacu\cdot \mu^{cs}_{x_0}
	\]
	where
	\begin{equation}\label{eq:jacu}
		\Jacu(x)=\prod_{j=1}^{\oo}\frac{e^{\varphi\circ f^{-j}(\hu x)}}{e^{\varphi\circ f^{-j}(x)}}.
	\end{equation} 
	
\begin{proof}[sketch of the proof]
By $(2)$ we get that $f^{-n}\mu_{x_0}^{cs}=e^{\SB\circ f^{-n+1}-nP}\mu_{f^{-n}x_0}^{cs}$, and similarly for $y_0$. For $x\in A(x_0), y=\hu(x) \in B(y_0)$, the points $f^{-n}x, f^{-n}y$ approximate each other exponentially fast with $n$, therefore by the H\"older assumption on $\varphi$, the difference $|\SB(f^{-n+1}(x))-\SB(f^{-n+1}(x))|$ is bounded in $n$, thus implying that $\Jacu(x)$ is well defined. Finally, by invariance of the foliations, $\hu{x_0,y_0}=f^{n}\circ\hu_{f^{-n}x_0,f^{-n}y_0}\circ f^{n}$, and representing $\hu_{f^{-n}x_0,f^{-n}y_0}$ in exponential charts near $f^{-n}x_0$ one sees that this map approximates uniformly the identity. From the above facts follows the claim.   
\end{proof}

Similar considerations can be made to the $\mu^{cu}, \Fs$. We now define on each leaf a $\Wu{x}\in\Fu$ a projective class of measures $[\nu^u_x]$ where
\begin{equation}\label{eq.nux}
\nu^u_x=\Delta_x^u\ \mux\quad \Delta_x^u(y):=\prod_{k=1}^{\oo}\frac{e^{\varphi\circ f^{-k}(y)}}{e^{\varphi\circ f^{-k}(x)}}, y \in \Wu{x}.
\end{equation}
The function $\Delta_x^u:\Wu{x}\to\Real$ is continuous, and if $x'\in \Wu{x}$, $\nu^u_{x'}=\Delta^u_{x'}(x)\nu^u_x$. Furthermore, we have 
\begin{align}\label{eq:invariancianu}
f^{-1}\nux[fx]=\Delta_{fx}\circ f\cdot f^{-1}\mux[fx]=\Delta_{fx}\circ f \cdot e^{P-\varphi}\mux =e^{P-\varphi(x)}\nux.
\end{align}

\section{Construction of the equilibrium state} % (fold)
\label{sec:construction_of_the_equilibrium_state}
% section construction_of_the_equilibrium_state (end)

We will now use the families of measures given in theorem \ref{thm:familiasmedidas} to construct an equilibrium state for the system $(f,\varphi)$. We keep the assumptions of that theorem for the rest of the article.  

\begin{definition}
If $\F\subset M^n$ is a foliation of codimension $q$ we say that an open set $U\subset M$ is foliation box of $\F$ if it is homeomorphic to $C=(-1,1)^{n-q}\times (-1,1)^{q}$ by a homeomorphism sending $\F|U$ to the horizontal foliation of $C$. In this case the embedded discs of $U$ corresponding to the vertical foliation of $C$ will be called the vertical slices of $U$.
\end{definition}

Fix a foliation box $U$ of $\Fu$ together with a vertical slice $W$, that without loss of generality can be assumed to  be a disc in some $\Wcs{x_0}$. For $A\subset U$ open, define the function $\al_{U,W,A}:W\to \Real$ by $\al_{U,W,A}(w)=\nu^u_w(A\cap\Wu{w,U})$, where $\Wu{u,U}$ is the connected component of $\Wu{w}\cap U$ that contains $w$. This function is Borel measurable (in fact, semi-continuous).

Now consider the Borel regular measure $\mm_{U,W}$ on $U$ determined by
\[
	A\subset U\text{ open}\Rightarrow \mm_{U,W}(A)=\int_W  \al_{U,W,A}(w) \der\mu^{cs}_{x_0}(w).
\]
The key property is the following.

\begin{proposition}
 If $W'$ is another vertical slice of $U$ then  $\mm_{U,W}=\mm_{U,W'}$. 
\end{proposition}  

\begin{proof}
Without loss of generality assume $W' \subset \Wcs{x_0'}$, and consider $h=\hu_{w_0,w_0'}:W\rightarrow W'$ the Poincaré map. For $A \subset U$ open and $w\in U$ denote $A_w=A\cap \Wu{w,U}$; then 
\begin{align*}
\mm_{U,W'}(A)&=\int_{W'} \nu^u_{w'}(A_{w'})\der\mcsx[w_0'](w')=
\int_{W} \nu^u_{h(w)}(A_{h(w)})\Jacu[w_0,w_0'](w)\der\mcsx[w_0](w)\\
&=\int_{W} \left(\nu^u_{h(w)}(A_{w})\prod_{k=1}^{+\oo}\frac{e^{\varphi\circ f^{-k}h(w)}}{e^{\varphi\circ f^{-k}(w)}}\right)\der\mcsx[w_0](w)\\
&=\int_W \nu^u_w(A_w)\der\mcsx[w_0](w)=\mm_{U,W}(A).
\end{align*}
\end{proof}

We write $\mm_{U}=\mm_{U,W}$ where $W \subset U$ is any vertical slice. As a consequence, we have that if $U,U', U''$ are foliation boxes of $\Fu$ with $U\cup U' \subset U''$ then $\mm_{U}=\mm){U'}$. We can thus fix a finite covering $\{U_i\}_{i=1}^r$ of foliation boxes of $\Fu$ and define a Borel probability measure $\muf$ on $M$ such that for every $i$ it holds $\muf|U_i=c\cdot\mm_{U_i}$, where $c>0$ is a normalization constant.

It is not hard to realize that $\muf$ is $f$-invariant (compare $(1)$ and $(2)$ in \ref{thm:familiasmedidas}). The following is also true.

\begin{theorem}\label{thm:estadodequilibrio}\hfill
\begin{enumerate}
	\item $\muf$ is the unique equilibrium state for the potential $\varphi$ and $P=\Ptop(f,\varphi)$; in particular $\muf$ is ergodic.
	\item If $U$ is a sufficiently small foliation box centered at $x\in M$, then $\muf|U$ is equivalent to $\mux\times \mcsx$.
	\item There exists $K\geq 0$ only depending on $\varphi$ such that for every $\epsilon>0$ there exists $c(\ep)>0$  satisfying for every $x\in M,n\geq 0$
	\[
	c(\epsilon)^{-1}e^{-Kn}\leq\frac{\muf(D(x,\epsilon,n))}{e^{\SB(x)-n\Ptop(\varphi)}}\leq e^{Kn}c(\epsilon).
		\]
	If the potential is constant along leaves of $\Fc$ then one can take $K=0$.  
\end{enumerate}
\end{theorem}

Uniqueness implies ergodicity, by standard arguments. The last part is a generalization of the so called Gibbs property, a concept of central importance in statistical physics. See \cite{thermodynform}.

\section{Characterization of equilibrium states in terms of conditional measures} % (fold)
\label{sec:characterization_of_equilibrium_states_in_terms_of_conditional_measures}

In hyperbolic systems (say, taking $E^c=0$ in definition \ref{def:centerisometry}), a particularly relevant measure is the SRB measure. This is an $f$-invariant measure $\mu$ that can be characterized by one of the two following equivalent conditions:
\begin{enumerate}
 	\item $\mu$ is an equilibrium state for the the potential $\varphi=-\log \det Df|E^u$;
 	\item the conditionals of $\mu$ induced in leaves of $\Fu$ are absolutely continuous with respect the induced Lebesgue.
\end{enumerate}
This notion was introduced first by Sinai in \cite{SinaiMarkov} and developed by Ruelle and Bowen \cite{SRBattractor}  \cite{BowenMarkov}, \cite{BowenRuelle}. It is difficult to overestimate the importance of SRB measures in ergodic theory; the survey article \cite{whatareSRB} remains as an excellent introduction to the subject, and the reader can also check \cite{Beyond} for more recent developments.

Let us explain the second condition. Given a Borel probability $\mm$ on $M$ and a family of measures $\eta^u\{\eta^u_x\}_{x\in M}$ on the leaves of $\Fu$ one says that $\mm$ has conditionals absolutely continuous/equivalent to $\eta^u$ if for some (equivalently, for every) measurable $\mm$-partition $\xi$ subordinated to the partition by leaves of $\Fu$ it holds that the induced conditional measures $\mm^{\xi}_x$ are absolutely continuous/equivalent with respect to $\eta^u_x$, for $\mm\aep$; in this part we use freely the theory of conditional measures as developed for example in \cite{Rokhlin}. Going back to $(2)$, for SRB measures the family $\eta^u$ is given by the corresponding induced Lebesgue measures on leaves of $\Fu$. Let us also point out that this characterization of SRB measures for (non-hyperbolic) systems is consequence of the very general results of Ledrappier and Young \cite{LedYoungI,LedYoungII}, a culmination of several other important contributions that for the sake of keeping this presentation short, we refer to the above cited articles for the references. 

While trying to generalization the characterization to other equilibrium states, the difficulty is the absence of families of measures to compare with. Here we solve this problem, that even for classical hyperbolic system remained unknown until now (observe that if $E^c=0$ then any potential is automatically constant along center leaves).

\begin{theorem}\label{thm:conversemarginal}
Let $\mm\in \PTM{f}{M}$ and assume that with respect to some\footnote{Of the type considered in \cite{LedYoungI}.} partition $\xi$ we have that $\mm_x<<\mu_x^u$ for $\mm\aep(x)$. Then $\mm$ is an equilibrium estate (a posteriori, $\mm=\muf$).
\end{theorem}

\section{Bernoulli property of the equilibrium state} % (fold)
\label{sec:bernoulli_property_of_the_equilibrium_state}

We can say more about the statistical properties of the system $(f,\muf)$ besides its ergodicity. A much stronger property is the so called Kolmogorov property, that can be characterized as follows.
\begin{definition}
For $\mm\in \PTM{f}{M}$ we say that $(f,\mm)$ is K-system if its Pinsker $\sigma$-algebra $\mathrm{Pin}(f,\mm)$  is trivial, where $\mathrm{Pin}(f,\mm)$ is generated by 
\[\{\xi \text{ countable measurable partition }: H_{\mm}(\xi)<\oo, h_{\mm}(f;\xi)=0 \}.
\] 
\end{definition}
Above $h_{\mm}(f;\xi)$ corresponds to the entropy of $f$ in the partition $\xi$; see \cite{ErgWal}.

There is a useful characterization of $\mathrm{Pin}(f,\mm)$. Given $X \subset M$ we say that it is 
\begin{itemize}
	\item $s$-saturated ($u$-saturated) if $x\in X\Rightarrow \Ws{x}\subset X$ (resp. $\Ws{u}\subset X$);
	\item bi-saturated if it is both $s$ and $u$ saturated;
	\item $\mm$-essentially $s$-saturated ($u$-saturated) if there exists a $s$-saturated ($u$-saturated) Borel set $X_0$ so that $\mm(X\Delta X_0)=0$.
\end{itemize}

\begin{theorem}[Ledrappier-Young, \cite{LedYoungI}]
	If $X\in \mathrm{Pin}(f,\mm)$ then $X$ is both $\mm$-essentially $s$-saturated and $\mm$-essentially $u$-saturated.
\end{theorem} 

Our argument to show the $K$-property is based on the program to establish stable ergodicity for conservative systems, developed originally by Grayson, Pugh and Shub \cite{StErgGeo} and extended by Pugh, Shub \cite{StableJulienne}, and several other authors, in particular Burns, Wilkinson \cite{ErgPH} and Hertz, Hertz and Ures \cite{AccStaErg}. These methods are very geometrical, but rely in the properties of the Lebesgue measure (in particular, its conditionals on leaves of the associated invariant foliations); in our case the corresponding measures $\mu^u,\mu^{cs}$ are more difficult to deal with, so a non-trivial adaptation is necessary. 

Fix then $X$ that is $\mm$-essentially $s$-saturated: we will show that it coincides with a $s$-saturated set $D(X)$ that corresponds to the density points of some dynamically defined differentiation basis. 

Choose small numbers $0<\varepsilon, \sigma<1$ and define the $\ast$-juliennes as  
\begin{align*}
J^u_n(x)&:=f^{-n}(\Wu{f^nx;\varepsilon})\\
J^s_n(x)&:=f^n(\Ws{f^{-n}x;\varepsilon})\\
B^c_n(x)&:=\Wc{x;\sigma^n}\\
J^{cu}_n(x)&:=\bigcup_{\mathclap{y\in B^c_n(x)}}\ J^u_n(y)\\
J^{scu}_n(x)&:=\bigcup_{\mathclap{y\in J_{n}^{cu}(x)}}\ J^s_n(y).
\end{align*}

By a careful control of the sizes of the juliennes we are able to deduce that $J^{scu}=\{J^{scu}_n(x):x\in M, n\in\Nat\}$ is a Vitali differentiation basis (cf. \cite{Guzman1975}) and the following.

\begin{theorem}
For every Borel set $X \subset M$, the set $D(X)$ of its density points with respect to the basis $J^{scu}$ is coincides $\muf\aep$ with $X$, and is $s$-saturated. Therefore, any $\muf$-essentially $s$-saturated set coincides $\muf\aep$ with a $s$-saturated set. 
\end{theorem}

The above Theorem implies that $(f,\muf)$ is a Kolmogorov system, due to the fact that every leaf of $\Ws$ is dense (minimality of $\Fs$): given $X\in \mathrm{Pin}(f,\muf)$ the sets $D(X), D(X^c)$ are $s$-saturated, therefore if $0<\muf(X)<1$ we will have density points of $X$ and $X^c$ arbitrarily close, contradicting their definition.

After establishing the $K$-property, we improve the result and show that in fact the system $(f,\muf)$ is isomorphic to a Bernoulli scheme.

\begin{definition}
Let $\mm\in \PTM{f}{M}$. The system $(f,\muf)$ is Bernoulli if its induced stochastic process is isomorphic to the Bernoulli process with finite marginal distribution. 
\end{definition}

Equivalently, it is measure theoretically isomorphic to a map $\sigma: \{1,\cdots, N\}^{\Z}\to \{1,\cdots, N\}^{\Z}$,
$\sigma(\{x_n\}_n)=\{x_{n+1}\}_n$, where the $\sigma$-invariant measure is the product measure induced by some finite distribution on $\{1,\cdots, N\}$.

Using the Kolmogorov property we are able to adapt the arguments of Ornstein and Weiss \cite{GeoBernoulli} for the time-one of the geodesic flow to our setting and prove:

\begin{theorem}
	$(f,\muf)$ is Bernoulli.
\end{theorem}

Again, a non-trivial amount of work is required since Ornstein and Weiss method is adapted for the Lebesgue measure, whereas in our case the control in the conditionals is necessarily more delicate.

% section bernoulli_property_of_the_equilibrium_state (end)

\section{Applications to the rank-one case: unique ergodicity of quasi-invariant measures} % (fold)
\label{sec:applications_to_the_rank_one_case_unique_ergodicity_of_quasi_invariant_measures}

Our results are a generalization of the classical theorems for hyperbolic diffeomorphisms and flows, \cite{SRBattractor}, that is, rank-one Anosov actions (of $\Z$ or $\Real$). In these cases the results can be obtained using the powerful tool of symbolic dynamics, which permits to reduce the study of the smooth map to a much more manageable symbolic model. Regrettably, this technology becomes much more intricate outside hyperbolic systems, and although Sarig made some recent breakthrough in establishing symbolic models for non completely hyperbolic systems \cite{Sarig2012}, the tools become more difficult to work and seem to require non-uniform hyperbolicity (which is never satisfied for center isometries). 

In spite of the applicability of symbolic dynamics to the study of hyperbolic systems, our geometrical method gives some new information even in this classical setting. We will enunciate one illustrative result.

Suppose that $f:M\to M$ is an hyperbolic diffeomorphism such that $\dim E^u=1$, which with no loss of generality can be assumed to be oriented. Then $E^u$ is tangent to the orbits of a flow $\Phi^u=\{\Phi^u_t:M\to M\}_{t\in \Real}$, called the horocyclic flow. This is a prototype of parabolic flow; the dynamics of such systems have several consequences in geometry and number theory; see \cite{UnipotentFlows} for a discussion.

The following celebrated theorem is originally due to Furstenberg \cite{Furstenberg1973}, while the version below is due to Marcus \cite{Marcus1975}.  

\begin{theorem}[Furstenberg]
If $f$ is a hyperbolic map of class $\mathcal{C}^2$ such that every orbit of $\Phi^u$ is dense, then $\Phi^u$ is uniquely ergodic. That is, there exists only one (probability) measure invariant under $\Phi^u$. 

It follows that every orbit of $\Phi^u$ is equidistributed in $M$.
\end{theorem}

We remark that no example of hyperbolic map not satisfying the minimality condition is currently known. Above, a Borel measure $\mu$ is said to be invariant under flow if its invariant under every $\Phi^u_t$. A more general notion is the following.

\begin{definition}
Given a flow $\Psi=\{\psi_t\}_t$ on a $M$, a measure $\mu$ is conformal for $\Psi$ if there exists a family of positive functions $J=\{J_t:M\to\Real_{>0}\}$, called the Jacobian of $\Psi$ with respect to $\mu$ such that for every $t\in\Real$,
\[
	(\Phi_{-t}^u)_{\ast}\mu=J_t\mu.
\] 	
 \end{definition}
This definition was given by Patterson in \cite{Patterson1976}, in a different context (limit sets for Fuchsian groups), and has an important role in geometry and ergodic theory. See for example \cite{FeliksPrzytycki2011}, \cite{overviewpattersonsullivan}. The definition for diffeomorphisms is analogous.

The following is a remarkable existence and uniqueness result for conformal measures.
 
\begin{theorem}[Douady and Yoccoz \cite{Douady1999}]
Let $f:\Tor\to\Tor$ be a $\mathcal{C}^2$ diffeomorphism of the circle with irrational rotation number. Then for every $s\in \Real$ there exists a unique conformal measure with Jacobian $s\cdot Df$. 
\end{theorem}

A particularly natural Jacobian for the flow $\Phi^u$ is obtained by taking $J=\mathrm{Jac}^{u}$ (cf. \eqref{eq:jacu}). Using our methods we are able do prove the next theorem.

 \begin{theorem}\label{thm:uniqueergoq}
In the cited hypotheses, let $\varphi:M\to \Real$ be a H\"older function and consider the multiplicative cocycle $\Jacu$ that it defines. Then there exists a unique quasi-invariant measure for $\Phi^u$ with Jacobian $\mathrm{Jac}^{u}$.
 \end{theorem}

This generalizes Furstenberg's result, and shows some strong rigidity in the possible dynamically relevant measures for the horocyclic flow outside its invariant one. The above theorem has also a version for hyperbolic flows; for the particular case of the geodesic flow in an hyperbolic manifold, the previous result was first established by Babillot and Ledrappier \cite{geodesicbabillotledrappier}. See also Schapira's article \cite{SCHAPIRA2004}. 

Comparing with Douady and Yoccoz' result, one gets enough evidence for the existence of some rigidity phenomena in the set of invariant measures of parabolic systems. One can ask:

\begin{question}
Let $\Phi$ be a minimal parabolic flow and $J$ a non-negative multiplicative cocycle. Does there exist a unique conformal measure with Jacobian $J$?
\end{question}

A particular instance of parabolic flows are unipotent flows; due to the general results for such type of system obtained by Ratner (and others) \cite{Ratner1991}, a positive answer of the above setting would be very interesting.

\section{Concluding remarks} % (fold)
\label{sec:concluding_remarks}

In the present note we gave a resume of some new geometrical methods to study thermodynamic formalism for generalizations of hyperbolic systems. These method seem to be generalizable to other situations, and hopefully would shed some light outside the hyperbolic realm. We point out the existence of other geometrical approaches considering some particular cases, as \cite{Climenhaga2020} and \cite{Spatzier2016}. The synergy between these and our methods stands as an interesting problem to investigate.

All the reported results in this article are in preprint form in \cite{EqStatesCenter} and \cite{ContributionsErgodictheory}.

\smallskip

\paragraph{\textbf{Acknowledgments}} The authors thank Barbara Schapira for bringing to our attention the references regarding theorem \ref{thm:uniqueergoq} for the geodesic flow case.

% section concluding_remarks (end)

\bibliographystyle{alpha}
\bibliography{/home/pablo/Dropbox/Latex/bibliografia.bib}

\begin{thebibliography}{CRH21b}

\bibitem[BDV05]{Beyond}
C.~Bonatti, L.~D\'iaz, and M.~Viana.
\newblock {\em {Dynamics Beyond Uniform Hyperbolicity}}, volume 102 of {\em
  Encyclopaedia of Mathematical Physics}.
\newblock Springer-Verlag, 2005.

\bibitem[BK83]{brinkatok}
M.~Brin and A.~Katok.
\newblock On local entropy.
\newblock In Springer, editor, {\em Proc. Conf. in Dyn. Sys, Rio de Janeiro},
  volume 1007 of {\em Lecture Notes in Mathematics}, pages 30--38, 1983.

\bibitem[BL98]{geodesicbabillotledrappier}
M.~Babillot and F.~Ledrappier.
\newblock Geodesic paths and horocycle flow on abelian covers.
\newblock In S.~G. Dani, editor, {\em Proceedings of the International
  Colloquium on Lie Groups and Ergodic Theory, Mumbai, 1996}, New Delhi, 1998.
  Published for the Tata Institute of Fundamental Research by Narosa Pub. House
  International distribution by American Mathematical Society.

\bibitem[Bow70]{BowenMarkov}
R.~Bowen.
\newblock {Markov Partitions for Axiom A Diffeomorphisms}.
\newblock {\em American Journal of Mathematics}, 92(3):pp. 725--747, 1970.

\bibitem[Bow08]{EquSta}
R.~Bowen.
\newblock {\em {Equilibrium States and the Ergodic Theory of Anosov
  Diffeomorphisms}}, volume 470 of {\em {Lect. Notes in Math.}}
\newblock Springer Verlag, 2008.

\bibitem[BR75]{BowenRuelle}
R.~Bowen and D.~Ruelle.
\newblock The ergodic theory of {A}xiom {A} flows.
\newblock {\em Invent. Math.}, 29:181--202, 1975.

\bibitem[BW10]{ErgPH}
K.~Burns and A.~Wilkinson.
\newblock {On the Ergodicity of Partially Hyperbolic Systems}.
\newblock {\em Ann. of Math}, 171:451--489, 2010.

\bibitem[CPZ20]{Climenhaga2020}
V.~Climenhaga, Ya. Pesin, and A.~Zelerowicz.
\newblock Equilibrium measures for some partially hyperbolic systems.
\newblock {\em Journal of Modern Dynamics}, 16(0):155--205, 2020.

\bibitem[CRH21a]{ContributionsErgodictheory}
P.~D. Carrasco and F.~Rodriguez-Hertz.
\newblock Contributions to the ergodic theory of hyperbolic flows: unique
  ergodicity for quasi-invariant measures and equilibrium states for the
  time-one map.
\newblock {\em arXiv 2103.07333}, 2021.

\bibitem[CRH21b]{EqStatesCenter}
P.~D. Carrasco and F.~Rodriguez-Hertz.
\newblock Equilibrium states for center isometries.
\newblock {\em arXiv 2103.07323}, 2021.

\bibitem[dG75]{Guzman1975}
Miguel de~Guzm{\'{a}}n.
\newblock {\em Differentiation of integrals in {Rn}}.
\newblock Springer Berlin Heidelberg, 1975.

\bibitem[DY99]{Douady1999}
R.~Douady and J.~C. Yoccoz.
\newblock {Nombre de Rotation des Diffeomorphismes du Cercle et Mesures
  Automorphes}.
\newblock {\em Regular and Chaotic Dynamics}, 4(4):19, 1999.

\bibitem[Esk07]{UnipotentFlows}
A.~Eskin.
\newblock {Unipotent Flows and Applications}.
\newblock In {\em Homogeneous Flows, Moduli Spaces and Arithmetic}, pages
  71--130. American Mathematical Society, 2007.

\bibitem[FP11]{FeliksPrzytycki2011}
Mariusz~Urbanski Feliks~Przytycki.
\newblock {\em {Conformal Fractals - Ergodic Theory Methods}}.
\newblock Cambridge University Press, 2011.

\bibitem[Fur73]{Furstenberg1973}
H.~Furstenberg.
\newblock The unique ergodigity of the horocycle flow.
\newblock In {\em Lecture Notes in Mathematics}, pages 95--115. Springer Berlin
  Heidelberg, 1973.

\bibitem[GPS94]{StErgGeo}
M.~Grayson, C.~Pugh, and M.~Shub.
\newblock {Stably Ergodic Diffeomorphisms}.
\newblock {\em Annals of Mathematics}, 140(2):295--329, 1994.

\bibitem[HHU07]{PartSurv}
F.~Rodriguez Hertz, M.~Rodriguez Hertz, and R.~Ures.
\newblock {A Survey of Partially Hyperbolic Dynamics}.
\newblock In Mikhail Lyubich Charles Pugh Michael~Shub {Giovanni Forni},
  editor, {\em {Partially Hyperbolic Dynamics , Laminations and Teichm{\"u}ller
  Flow}}, volume~51 of {\em {Fields Institute Communications}}, pages 35--88,
  2007.

\bibitem[HHU08]{AccStaErg}
F.~Rodriguez Hertz, M.~A.~Rodriguez Hertz, and R.~Ures.
\newblock Accessibility and stable ergodicity for partially hyperbolic
  diffeomorphisms with 1{D}-center bundle.
\newblock {\em Invent. Math.}, 172:353--381, 2008.

\bibitem[KS96]{Katok1996}
A.~Katok and R.~J. Spatzier.
\newblock Invariant measures for higher-rank hyperbolic abelian actions.
\newblock {\em Ergodic Theory and Dynamical Systems}, 16(4):751--778, 1996.

\bibitem[LY85a]{LedYoungI}
F.~Ledrappier and L.-S. Young.
\newblock The metric entropy of diffeomorphisms. {I}. {C}haracterization of
  measures satisfying {P}esin's entropy formula.
\newblock {\em Ann. of Mathematics}, 122:509--539, 1985.

\bibitem[LY85b]{LedYoungII}
F.~Ledrappier and L.-S. Young.
\newblock The metric entropy of diffeomorphisms. {I}{I}. {R}elations between
  entropy, exponents and dimension.
\newblock {\em Ann. of Mathematics}, 122:540--574, 1985.

\bibitem[Mar75]{Marcus1975}
B.~Marcus.
\newblock Unique ergodicity of the horocycle flow: Variable negative curvature
  case.
\newblock {\em Israel Journal of Mathematics}, 21(2-3):133--144, 1975.

\bibitem[Mar04]{TesisMarg}
G.~A. Margulis.
\newblock {\em On some aspects of the theory of {A}nosov systems}.
\newblock Springer, 2004.

\bibitem[OW73]{GeoBernoulli}
D.~Ornstein and B.~Weiss.
\newblock Geodesic flows are {B}ernoullian.
\newblock {\em Israel Journal of Mathematics}, 14(2):184--198, 1973.

\bibitem[Pat76]{Patterson1976}
S.~J. Patterson.
\newblock The limit set of a {F}uchsian group.
\newblock {\em Acta Mathematica}, 136(0):241--273, 1976.

\bibitem[PS00]{StableJulienne}
C.~Pugh and M.~Shub.
\newblock {Stable Ergodicity and Julienne Quasi-Conformality}.
\newblock {\em JEMS}, 2(1):1--52, 2000.

\bibitem[Qui06]{overviewpattersonsullivan}
J.~F. Quint.
\newblock {An overview of Patterson-Sullivan theory}.
\newblock In {\em The barycenter method, FIM, Zurich}, 2006.

\bibitem[Rat91]{Ratner1991}
Marina Ratner.
\newblock Raghunathan’s topological conjecture and distributions of unipotent
  flows.
\newblock {\em Duke Mathematical Journal}, 63(1):235--280, jun 1991.

\bibitem[Rok62]{Rokhlin}
V.~Rokhlin.
\newblock {On the Fundamental Ideas of Measure Theory}.
\newblock {\em Transl. Amer. Math. Soc.}, 10:1--52, 1962.

\bibitem[Rue76]{SRBattractor}
D.~Ruelle.
\newblock A measure associated with {A}xiom-{A} attractors.
\newblock {\em American Journal of Mathematics}, 98(3):619, 1976.

\bibitem[Rue04]{thermodynform}
D.~Ruelle.
\newblock {\em {Thermodynamic formalism. The mathematical structures of
  equilibrium statistical mechanics}}.
\newblock Cambridge University Press, Cambridge, second edition edition, 2004.

\bibitem[Sar12]{Sarig2012}
Omri~M. Sarig.
\newblock Symbolic dynamics for surface diffeomorphisms with positive entropy.
\newblock {\em Journal of the American Mathematical Society}, 26(2):341--426,
  nov 2012.

\bibitem[Sch04]{SCHAPIRA2004}
B.~Schapira.
\newblock On quasi-invariant transverse measures for the horospherical
  foliation of a negatively curved manifold.
\newblock {\em Ergodic Theory and Dynamical Systems}, 24(1):227--255, 2004.

\bibitem[Sin68]{SinaiMarkov}
Ya. Sinai.
\newblock Markov partitions and {C}-diffeomorphisms.
\newblock {\em Functional Analysis and Its Applications}, 2(1):61--82, 1968.

\bibitem[SV16]{Spatzier2016}
R.~Spatzier and D.~Vischer.
\newblock Equilibrium measures for certain isometric extensions of {A}nosov
  systems.
\newblock {\em Ergodic Theory and Dynamical Systems}, 38(3):1154--1167, 2016.

\bibitem[Wal82]{ErgWal}
P.~Walters.
\newblock {\em {An introduction to Ergodic Theory}}.
\newblock Springer, 1982.

\bibitem[You02]{whatareSRB}
L.~S. Young.
\newblock What are {SRB} measures, and which dynamical systems have them?
\newblock {\em Journal of Statistical Physics}, 108(5):733--754, 2002.

\end{thebibliography}
\end{document}